\documentclass[11pt,a4paper,reqno]{amsart}
\usepackage{amsmath}
\usepackage{fullpage}
\usepackage{amssymb}
\usepackage{amsmath}
\usepackage{amsfonts}
\usepackage{amssymb}
\usepackage{amsthm}
\usepackage{mathtools}
\usepackage{color}
\usepackage{amsaddr}

\newtheorem{theorem}{Theorem}
\newtheorem{lemma}[theorem]{Lemma}
\newtheorem{claim}[theorem]{Claim}
\newtheorem{proposition}[theorem]{Proposition}

\newcommand{\Ppo}[1]{\Pr\left[#1\right]}
\newcommand{\Exp}[1]{\mathbb{E}\left[#1\right]}

\newcommand\bb[1]{\bigl(#1\bigr)}

\renewcommand\Pr{\mathop{\mathbb P{}}\nolimits}
\newcommand\Var{{\mathop{\mathrm{Var}}\nolimits}}
\newcommand\eps{\varepsilon}

\newcommand\upto{\nearrow}
\newcommand\la{\lambda}

\newcommand\plow{\la_1-\delta,\la_2-\delta}
\newcommand\pmed{\la_1-\xi,\la_2-\xi}
\newcommand\phigh{\la_1,\la_2}
\newcommand\porig{\la_1,\la_2}
\newcommand\Prl[2]{\Pr_{#1}\left[#2\right]}
\newcommand{\Xl}[1]{X_{#1}}
\newcommand\Prli[1]{\Pr_{#1}}
\newcommand\br{{\mathcal B}}
\newcommand\ca{{\mathcal A}}

\newcommand\cD{{\mathcal D}}
\newcommand\cT{{\mathcal T}}

\newcommand\tG{{\widetilde G}}
\newcommand\whp{whp}
\newcommand\ev{{\mathcal E}}
\newcommand\el{{\mathcal L}}
\newcommand\rb{{\mathcal R}}
\newcommand\of{\circ}

\newcommand\Po{\operatorname{Po}}

\newcommand\Es[1]{\operatorname{\mathbb E}_{#1}}
\newcommand\diff{\bigtriangleup}

\newcommand\cc{{\mathrm c}}

\begin{document}

\title{The size of the giant joint component\\ in a binomial random double graph}
\author{Mark Jerrum\textsuperscript{*,\dag}}
\email{m.jerrum@qmul.ac.uk}
\address{School of Mathematical Sciences\\
	Queen Mary, University of London\\
	Mile End Road\\
	London E1 4NS, UK\\
}
\thanks{\textsuperscript{*} Supported by EPSRC grant EP/N004221/1}
\thanks{\textsuperscript{\dag} Supported by EPSRC grant EP/S016694/1}

\author{Tam\'as Makai\textsuperscript{*,\ddag}}
\email{tamas.makai@unito.it/t.makai@unsw.edu.au}
\address{Department of Mathematics G. Peano\\
	University of Torino\\
	Via Carlo Alberto, 10\\
	10123, Torino, Italy\\
}
\address{School of Mathematics and Statistics\\
	University of New South Wales\\
	Sydney, NSW, 2052, Australia
}
\thanks{\textsuperscript{\ddag} Supported by ``Memory in Evolving Graphs" (Compagnia di San Paolo/Universit\`a degli Studi di Torino) and ARC Grant DP190100977.}

\begin{abstract}
We study the joint components in a random `double graph' that is obtained by superposing red and blue binomial random graphs on $n$~vertices.
A joint component is a maximal set of vertices that supports both a red and a blue spanning tree.
We show that there are critical pairs of red and blue edge densities at which a giant joint component appears.  In contrast to the standard binomial graph model, the phase transition is first order:  the size of the largest joint component jumps from $O(1)$ vertices to $\Theta(n)$ at the critical point.  We connect this phenomenon to the properties of a certain bicoloured branching process.  
\end{abstract}

\maketitle

\section{Introduction}

In recent years there has been a growing interest in `multilayer networks' as a model for large real-world structures~\cite{Multilayer}.  Attention is focused on properties of a multilayer network that arise from interactions between the layers.  In the language of graph theory, we can treat a multilayer network as a collection of graphs, all sharing a common vertex set.   
The simplest case is a 
\emph{double graph} $G=(V,E_1,E_2)$ formed by superposing two graphs $G_1=(V,E_1)$ and $G_2=(V,E_2)$ over the same vertex set.  We refer to $E_1$ as the set of red edges and $E_2$ as the set of blue edges.  We are particularly interested in the random double graph $G(n,p_1,p_2)$ in which $G_1$ and $G_2$ are independent binomial (or Erd\H os-R\'enyi) random graphs on $[n]$, with edge probabilities $p_1$ and~$p_2$, respectively.  Thus a red edge is present between a given pair of vertices with probability~$p_1$, independently of all the other potential red and blue edges, and similarly for the blue edges.

The most intensively studied phenomenon in the theory of random graphs is the emergence and growth of a `giant component' in a binomial random graph as the edge probability increases~\cite[Chap.~5]{tombstone}.  A \emph{giant component} is a connected component that contains a constant fraction of the vertices, which appears at a certain critical edge probability, specifically $p=1/n$.  For $p=c/n$ with $c>1$, 
there is a unique giant component:  
all other connected components have size $O(\log n)$ with high probability\footnote{A property holds whp if the probability of it occurring tends to~1 as $n$ tends to infinity.} (whp). When $c<1$ there is no giant component.  This phase transition phenomenon is now understood in great detail~\cite{giant}.
A natural extension of this line of work to double graphs is the following.  
A joint component is a maximal set of vertices that supports both a red and a blue spanning tree.
(Note that the joint components form a partition of the vertex set of a double graph.)  We ask whether the largest joint component --- the potential \emph{giant} joint component --- undergoes a phase transition and, if so, what is the nature of the transition.  Note that the giant joint component is not simply the intersection of the red and blue giant components considered in isolation, though it is contained in the intersection.

The question of the existence of a giant joint component in a double graph was examined, in a slightly disguised form,  by Buldyrev, Parshani, Paul, Stanley and Havlin \cite{BPPSH10}.  The relevant scaling to use is $p_1=\la_1/n$ and $p_2=\la_2/n$, for constants $\la_1,\la_2$.  Buldyrev et al.\ provided a heuristic argument that the giant joint component appears at certain critical values of the pair $(\la_1,\la_2)$, and confirmed this predicted behaviour experimentally. Independently, Molloy \cite{MR3252926} provided a rigorous proof for the size of the giant joint component in the special case when $\la_1=\la_2$, and stated what should be the generalisation to unequal edge densities and even to graphs formed from three or more distinguished edge sets.  Both the heuristic and rigorous results approach the giant joint component from above, by repeatedly stripping vertices that cannot be contained in it. 

In this paper we take a very different approach to analysing the joint components of the double graph $G(n,\la_1/n,\la_2/n)$.
We show that for any $\lambda_1,\lambda_2\in \mathbb{R}$ whp any non-trivial joint component\footnote{A trivial joint component has size 1} contains exactly two vertices or a linear fraction of the vertices. In addition, whp there can be at most one joint component of linear size, which we call the giant joint component, or the \emph{joint-giant} for short. We establish the size of the joint-giant
as a function of $\la_1,\la_2$.   Interestingly, whereas the phase transition of a classical binomial random graph is second order, the phase transition in a double graph turns out to be first order.  Thus, if we plot the size of the largest component (scaled by $1/n$) in a binomial graph $G(n,\la/n)$ as a function of~$\la$, the resulting curve is continuous;  the phase transition is marked only by a discontinuity of the derivative at $\la=1$.  In contrast, for a double graph, the plot of the size of largest joint component is discontinuous at pairs $(\la_1,\la_2)$ lying on a curve~$C$ to be defined presently.  For example, as noted by Molloy, there is a critical value $\la^*=2.4554+$ such that when $\la_1=\la_2<\la^*$ there is no joint-giant, and when $\la_1=\la_2>\la^*$ there is a joint-giant of linear size.  In fact, when $\lambda_1=\lambda_2$ is just above $\lambda^*$ the joint-giant contains about $0.5117\,n$ vertices.  The curve $C$ defining the phase transition as a function of $\la_1$ and $\la_2$ is plotted in Figure~\ref{fig:phasediag}.  Above the curve, there is a unique joint-giant of linear size;  below, the largest joint component has size at most~2, whp.  These analytical results are consistent with numerical findings reported by Buldyrev et al.~\cite{BPPSH10} and, of course, with the analytic result of Molloy~\cite{MR3252926}.

\begin{figure}[h]
\begin{center}
\includegraphics[width=5cm]{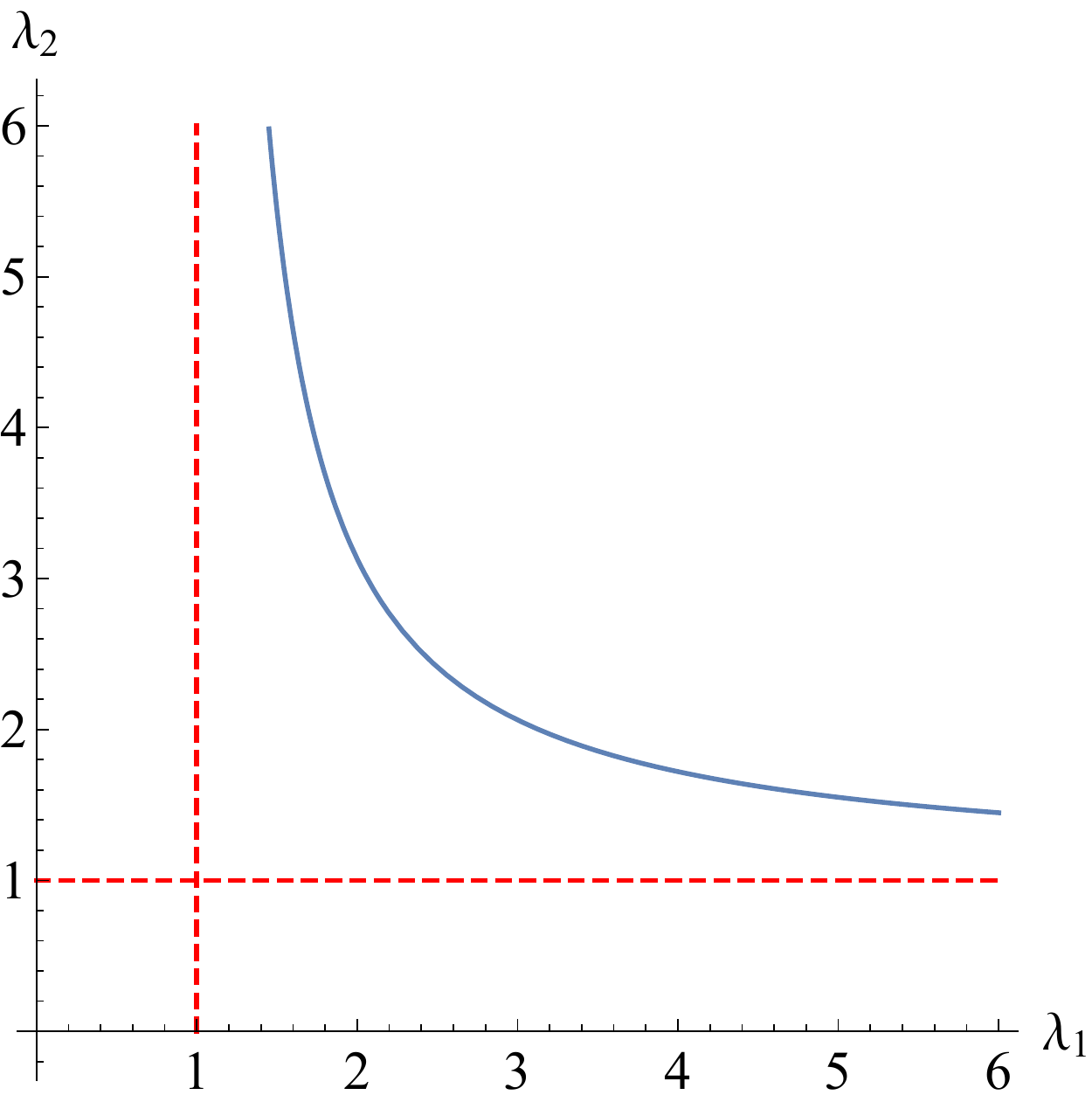}
\end{center}
\caption{The phase transition threshold plotted as a function of $\la_1$ and $\la_2$.  This is the curve $C$ from Theorem~\ref{thm:main}.}
\label{fig:phasediag}
\end{figure}

A superficially similar percolation model is jigsaw percolation introduced by Brummitt, Chatterjee, Dey and Sivakoff \cite{MR3349000}. This model is also defined on a double graph; however, in this case a bottom-up approach is used. Initially every vertex is in its own part and in every step of the process two parts are merged if there is a red and a blue edge between them. Several papers have been devoted to investigating when the process percolates, i.e., when every vertex is contained in the same part by the end of the process. So far the combination of various deterministic graphs with a binomial random graph \cite{MR3349000,GravnerSivakoff17} and the combination of two binomial random graphs \cite{BollobasRiordanSlivkenSmith17, CKM18} has been studied. In addition, extensions to multi-coloured random graphs \cite{CG17} and random hypergraphs \cite{BCKK17} exist.  

As hinted at earlier, one approach to locating the joint-giant is to repeatedly remove the vertices found in any small red and blue component of the graph.  Molloy \cite{MR3252926} analysed this process in order to establish the size of the joint-giant.  
Our method differs as we relate the size of the joint-giant in the double graph $G(n,\la_1,\la_2)$ to a bicoloured branching process, where every particle in the process has $\Po(\la_1)$ red offspring and independently $\Po(\la_2)$ blue offspring.  A joint-giant exists precisely when there is a positive probability that such a branching process contains an infinite red-blue binary tree, i.e., one in which every particle has one red offspring and one blue offspring.  In a sense, what we do is the opposite of the earlier approach, in that we are exploring the joint-giant from within.  We feel that this approach gives additional insight into the phase transition phenomenon.  The two approaches mirror earlier work on the $k$-core of a random graph, with Pittel, Spencer and Wormald~\cite{MR1385386} approaching the $k$-core from above, and Riordan~\cite{MR2376426} from below.
  
Denote the coloured rooted unlabelled tree created by the above branching process by $\Xl{\porig}$ and the associated probability distribution by $\Prli{\porig}$.  In order to state our result, we need to make some preliminary observations about $\Xl{\porig}$.  The root of the tree is $v_0$.
When we say that a particle $x$ of the
branching process has a certain property, we mean that the process
consisting of~$x$ (as the new root) and its descendants has the
property. 

A binary red-blue tree of height~$d$ is a perfect binary tree of height~$d$, where every internal vertex has a red and a blue offspring.
Let $\br_d$ be the event that $\Xl{\porig}$ contains a binary red-blue tree of height~$d$ 
with $v_0$ as the root, and let $\br=\lim_{d\to\infty}\br_d$
be the event that $\Xl{\porig}$ contains an infinite binary red-blue tree
with root $v_0$.
Then $\Prl{\porig}{\br_0}=1$. Also,
each particle in the first generation of $\Xl{\porig}$ has property $\br_d$ with probability
$\Prl{\porig}{\br_d}$. As these events are independent
for different particles, the number of red and blue offspring in the first generation
with property~$\br_d$ has a Poisson distribution with mean
$\la_1\Prl{\porig}{\br_d}$ and $\la_2\Prl{\porig}{\br_d}$ respectively. Thus, 
$\Prl{\porig}{\br_{d+1}}=\Ppo{\Po(\la_1\Prl{\porig}{\br_d})>0}\Ppo{\Po(\la_2\Prl{\porig}{\br_d})>0}$.

Since $\Ppo{\Po(\la x) >0}$ is a continuous, increasing function of $x$ on $[0,1]$, it follows 
(e.g., from Kleene's fixed point theorem)
that $\Prl{\porig}{\br}=\lim_{d\to\infty}\Prl{\porig}{\br_d}$ is given by the maximum solution
$\alpha$ to the equation 
$$\alpha=\Ppo{\Po(\la_1 \alpha) >0}\Ppo{\Po(\la_2 \alpha) >0}.$$ 
(Maximality comes from $\Prl{\porig}{\br_0}=1$.)

We denote this solution by $\beta(\la_1,\la_2)$.
Let $C=\partial\{(\la_1,\la_2) \mid \beta(\la_1,\la_2)=0\}$ be the boundary of the zero-set of~$\beta$.
The main result of the paper is the following:

\begin{theorem}\label{thm:main}
For $(\porig)\in (\mathbb{R}^+)^2\setminus C$ the number of vertices in the largest joint component of $G(n,\la_1/n, \la_2/n)$ is $\beta(\la_1,\la_2)n+o_p(n)$ as $n\to\infty$.  When $\beta(\la_1,\la_2)>0$, this giant joint component is unique.
\end{theorem}

In addition, we show that \whp\ any non-trivial joint component of sublinear size contains exactly two vertices, i.e.\ it is a pair of vertices connected by a red and a blue edge.

\begin{theorem}\label{thm:aux}
For $(\porig)\in (\mathbb{R}^+)^2$ we have that in $G(n,\la_1/n, \la_2/n)$ \whp\ no joint component of size~$k$ exists for any $2<k=o(n)$.
\end{theorem}

\subsection{Proof outline}

Our proof is based on a method introduced by Riordan~\cite{MR2376426} in order to determine the size of the $k$-core of a graph. The key idea is to define a pair of events, which depend only on the close neighbourhood of a vertex in the double graph, more precisely on vertices which are at distance $o(\log{n})$. The distance of two vertices in the double graph is defined as the distance in the graph $G(V,E_1\cup E_2)$.
In addition, \whp\ every vertex for which the first event holds is contained in the joint-giant; however, \whp\ none of the vertices for which the second event fails is found in the joint-giant, allowing us to establish a lower and an upper bound on the set of vertices in the joint-giant. The result follows if the probabilities of the two events are close enough. 

Local properties are chosen because there is an effective coupling between the random double graph in the close neighbourhood of a vertex and the branching process described above, allowing us to transfer results from the branching process to the random graph.

For the upper bound we will choose the event that either the neighbourhood of~$v$ contains a short cycle, including the red-blue cycle of length two, or $\br_s$ holds for some appropriately chosen~$s$. We show that any vertex, which does not have either of these properties is outside of any non-trivial joint component (Claim~\ref{clm:gianthasB}). An upper bound on the size of the joint-giant follows by providing an estimate on the expected number of these vertices and the second moment method.

The lower bound requires significantly more attention. In this case we define the event $\ca$, which is essentially a robust version of $\br_s$. We show that \whp\ many vertices have property $\ca$ and in addition every vertex with property $\ca$ is the root of a red-blue binary tree of depth~$s$, where every leaf has property $\ca$ within the remainder of the graph (Lemma~\ref{lem:size}). 

Now consider the graph spanned by the vertices found in the union of these trees. When applied to the $k$-core this roughly translates into taking the union of $k$-regular trees of depth~$s$ where every leaf is also the root of a $k$-regular tree of depth~$s$ in the remaining graph. This already identifies almost every vertex within the $k$-core. However the situation is not as straightforward for joint-connectivity as even though every vertex is contained in a red and a blue connected subgraph of size at least $s$, there is no guarantee that the set contains a red and a blue spanning tree.
While small components may appear in the random graph, the previously described set (spanned by trees rooted at vertices with property~$\ca$), due to its special structure, avoids this obstacle, and thus any component within it must have size at least $n^{3/5}$ (Proposition~\ref{prop:size2}). We complete the proof of Theorem~\ref{thm:main} with a sprinkling argument to show that the graph spanned by this subset is connected in both the red and the blue graph.

Theorem~\ref{thm:aux} follows from a simple first moment argument.

\subsection{Organisation of the paper}
For a double graph $G$ on $n$ vertices let $U'(G)$ be the maximal subset of vertices of $G$ such that in the subgraph spanned by $U'$, denoted by $G[U']$, every vertex is found in both a red and a blue connected subgraph of size at least $n^{3/5}$. Note that $U'(G)$ is closed under union and thus well defined.
The key result for showing the lower bound on the size of the joint-giant is the following.

\begin{proposition}\label{prop:size2}
	For every $(\phigh)\in (\mathbb{R}^+)^2\setminus C$ we have
	$$|U'(G(n,\la_1/n,\la_2/n))|\ge \beta(\phigh)n+o_p(n).$$
\end{proposition}
Note that it is enough to consider $(\phigh)\in (\mathbb{R}^+)^2\setminus C$ with $\beta(\phigh)>0$ as otherwise the trivial lower bound $0$ already implies the statement. 
Sections \ref{sec:branching} and \ref{sec:bptographs} are devoted to proving this result for such a fixed pair $(\phigh)$. 
Once this has been achieved, 
Theorems \ref{thm:main} and~\ref{thm:aux} follow swiftly in Section~\ref{sec:proofs}.

\section{A branching process}\label{sec:branching}
In this section we analyse $X_{\phigh}$.  This will form an idealised model of the local structure of a random double graph. The model is adequate, since the random graph is locally tree-like.  In analysing the branching process we rely heavily on ideas introduced by Riordan~\cite{MR2376426}.  Later, in Section~\ref{sec:bptographs} we create a bridge from the branching process to random graphs.  First we need to show that the function $\beta$ is well behaved.  

\begin{lemma}\label{lem:cts}
The function $\beta$ is continuous in $(\mathbb{R}^+)^2\setminus C$.  
\end{lemma}
\begin{proof}

Define $h(\beta)=h_{\la_1,\la_2}(\beta)=(1-e^{-\la_1\beta})(1-e^{-\la_2\beta})-\beta$.  Recall that $\beta(\la_1,\la_2)$ is defined to be the maximum root $\beta$ of $h_{\la_1,\la_2}(\beta)=0$.  Clearly, $\beta=0$ is one root.   We are interested in identifying any non-zero roots. Note that, since $h(\beta)<0$ for all $\beta\geq1$, there are no roots with $\beta\ge 1$. Differentiating twice, we obtain
$$
h''_{\la_1,\la_2}(\beta)=e^{-(\la_1+\la_2)\beta}\big[(\la_1+\la_2)^2-\la_1^2e^{\la_2\beta}-\la_2^2e^{\la_1\beta}\big],
$$
which is positive up to a certain value of $\beta$ and then negative.  Coupled with $h(0)=0$ and $h'(0)=-1<0$, this implies that $h$ has at most two strictly positive roots. (Figure~\ref{fig:subandsuper} may assist in visualising the function $h(\beta)$.) Note that $h$ when considered as a function of $\la_1,\la_2,\beta$ is continuous in each of its variables, implying that any discontinuity is caused by the appearance of the first strictly positive root completing the proof.
\end{proof}

\begin{figure}[h]
\begin{minipage}{0.49\textwidth}
\centering
\includegraphics[width=6cm]{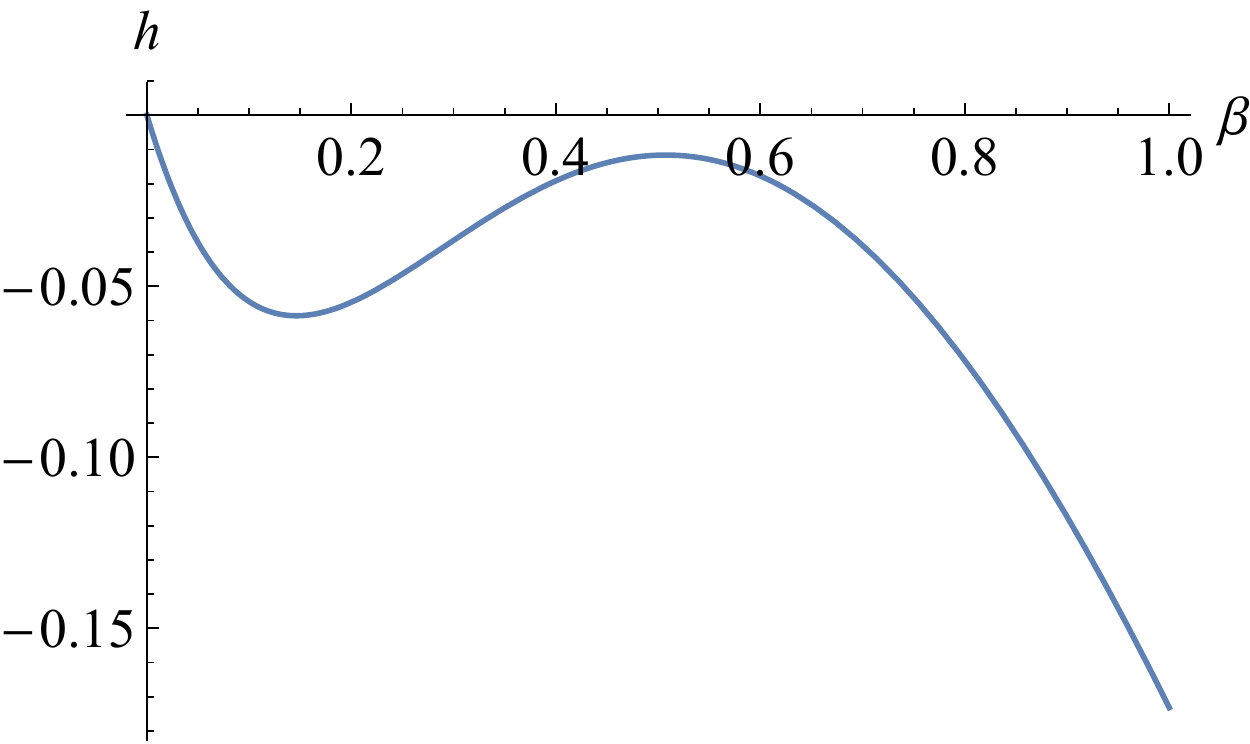}
\end{minipage}
\begin{minipage}{0.49\textwidth}
\centering
\includegraphics[width=6cm]{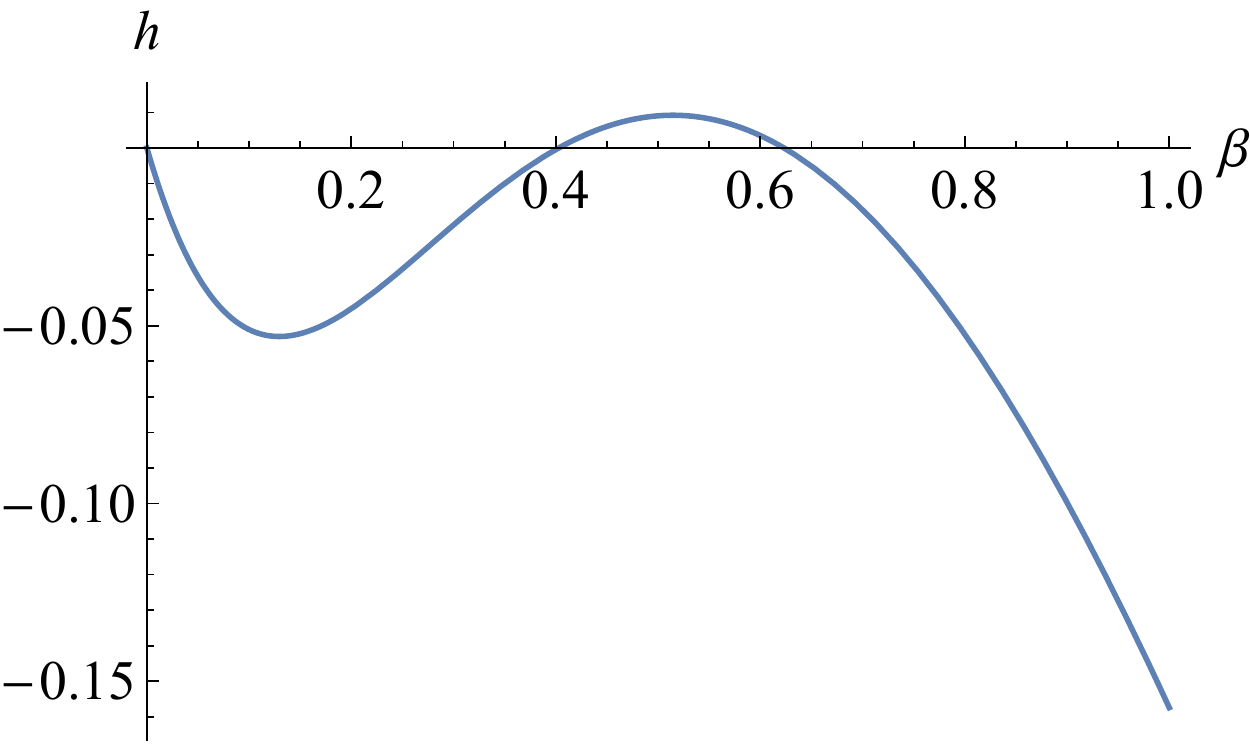}
\end{minipage}
\caption{The function $h(\beta)$ in the subcritical (left) and supercritical (right) regimes.}
\label{fig:subandsuper}
\end{figure}

Figure~\ref{fig:subandsuper} shows two plots of the function $h$ from the proof of Lemma~\ref{lem:cts}.  In the first, $\la_1=\la_2=2.4$ and there are no strictly positive roots, while in the second, $\la_1=\la_2=2.5$ and there are two, the maximum possible.  Between these two situations there is a critical value $\la^*$ such that setting $\la_1=\la_2=\la^*$ yields one positive root $\beta^*$.  Thus   Figure~\ref{fig:subandsuper} provides an informal pictorial explanation of the first order phase transition:  at $\la^*$ the maximum root jumps from 0 to~$\beta^*$. The behaviour of the function is similar when $\la_1\neq \la_2$.

For the rest of the section we assume that $(\phigh)\in (\mathbb{R}^+)^2\setminus C$ with $\beta(\phigh)>0$.

\begin{lemma}\label{lem:epsexist}
Suppose $(\phigh)\in (\mathbb{R}^+)^2\setminus C$ with $\beta(\phigh)>0$.  For $i=1,2$ the following holds: $\la_i\Ppo{\Po(\la_{3-i} \beta(\phigh))>0}>1$.
\end{lemma}

\begin{proof}
Without loss of generality let $i=1$. Assume for contradiction that
\begin{equation}\label{eq:contr}
\la_1 \Ppo{\Po(\la_{2} \beta(\phigh))>0}=\la_1 \left(1-\exp(-\la_{2} \beta(\phigh))\right)\le 1.
\end{equation}
Note that $\la_1\neq 0$ as $\la_1=0$ would imply $\beta(\phigh)=0$.
By the definition of $\beta(\phigh)$ we have
\begin{align*}
\beta(\phigh)&=\Ppo{\Po(\la_1 \beta(\phigh)) >0}\Ppo{\Po(\la_2 \beta(\phigh)) >0}\\
&=(1-\exp(-\la_1 \beta(\phigh)))(1-\exp(-\la_2 \beta(\phigh)))\\
&\stackrel{\eqref{eq:contr}}{\le} \frac{1-\exp(-\la_1 \beta(\phigh))}{\la_1},
\end{align*}
or equivalently
$$\exp(-\la_1 \beta(\phigh))\le 1-\la_1 \beta(\phigh).$$
Now this inequality holds only if $\la_1 \beta(\phigh)=0$ leading to a contradiction, as neither $\la_1$ nor $\beta(\phigh)$ is equal to 0.
\end{proof}

By Lemma~\ref{lem:epsexist} and since $x<e^{x-1}$ holds when $x>1$ we have, 
$$
\exp\left(\la_i \Ppo{\Po(\la_{3-i} \beta(\phigh))>0}\right)>e\la_i \Ppo{\Po(\la_{3-i} \beta(\phigh))>0},
$$
for $i=1,2$.  By continuity, there exists $\eps_0>0$ such that 
\begin{equation}\label{eq:epssmall}
\exp\left(\la_i \Ppo{\Po(\la_{3-i} (\beta(\phigh)-\eps))>0}\right)>e\la_i \Ppo{\Po(\la_{3-i} \beta(\phigh))>0},
\end{equation}
for $i=1,2$ and all $\eps\in(0,\eps_0]$.  Fix such an $\eps_0$ and an $\eps<\eps_0$.

Lemma~\ref{lem:cts} implies that there exists a $\delta>0$ such that 
\begin{equation}\label{eq:deltaspec}
\beta(\plow)>\beta(\phigh)-\eps
\end{equation}
and $\beta$ is continuous in 
 the closed ball of radius $\delta$ centred at $(\phigh)$, implying that 
for 
$i=1,2$ when $\xi \upto \delta$ we have
$$(\la_i-\xi) \beta(\pmed)\rightarrow (\la_i-\delta) \beta(\plow)<\la_i\beta(\plow).$$
Therefore there exists a $\xi\in(0,\delta)$ satisfying the following inequality for $i=1,2$
\begin{equation}\label{eq:xi}
\la_i \beta(\plow)>(\la_i-\xi) \beta(\pmed).
\end{equation}
Fix such a $\delta$ and $\xi$.

At this point we have fixed a number of parameters, which we collect together here for future reference:
\begin{itemize}
\item $\la_1,\la_2>0$ satisfy $\beta(\phigh)>0$;
\item $\eps_0>0$ is such that inequality \eqref{eq:epssmall} holds for all $\eps\in(0,\eps_0]$;
\item an $\eps\in(0,\eps_0]$;
\item $\delta>0$ satisfies inequality \eqref{eq:deltaspec}, and also that $\beta$ is continuous in the ball of radius $\delta$ centred at $(\phigh)$;
\item $\xi\in(0,\delta)$ satisfies inequality \eqref{eq:xi}.
\end{itemize}
\emph{These settings will remain in force until the end of Section~\ref{sec:bptographs}.}

Initially we will work on the branching process and then integrate these results into the random graph model.
If $\ev_1$, $\ev_2$ are properties of the branching process, with $\ev_1$ depending only on the
first $d$ generations, where the choice of $d$ will be implicit given the event $\ev_1$, let $\ev_1\of\ev_2$ denote the event that $\ev_1$ holds
if we delete all particles in generation $d$ of the branching process that do not have property $\ev_2$.
For example, with $\br_1$ the property of having at least one red and blue offspring, as above,
$\br_1\of\br_1=\br_2$, the property of having a red and a blue offspring each with at least one red and one blue offspring.

Let $\rb_k$ be the event that $\br_k$ holds in a {\em robust} manner, meaning
that $\br_k$ holds even after any particle in generation $k$ is deleted.
Since $\br=\br_k\of\br$, the event $\rb_k\of\br$ is the event that $\br$ holds 
even after deleting an arbitrary particle in generation $k$ and all of its descendants. 

\begin{lemma}\label{l1}
We have
\[
 \Prl{\plow}{\rb_k\of\br}\upto \beta(\plow)=\Prl{\plow}{\br}
\]
as $k\to\infty$.
\end{lemma}
\begin{proof}
Fix $0<\phi<1$. Note that $\Xl{(1-\phi)(\la_1-\delta),(1-\phi)(\la_2-\delta)}$ can be obtained by constructing $\Xl{\plow}$, and then deleting each edge (of the rooted tree)
independently with probability $\phi$, and taking for $\Xl{(1-\phi)(\la_1-\delta),(1-\phi)(\la_2-\delta)}$ the set of particles
still connected to the root.
To obtain an upper bound on the probability that $\br$ holds for $\Xl{(1-\phi)(\la_1-\delta),(1-\phi)(\la_2-\delta)}$, we use the above coupling and condition on $\Xl{\plow}$.

If $\br$ does not hold for $\Xl{\plow}$, it certainly does not hold for $\Xl{(1-\phi)(\la_1-\delta),(1-\phi)(\la_2-\delta)}$.  Furthermore, if $\br\setminus (\rb_k\of\br)$ holds for $\Xl{\plow}$, then there is a particle $v$ in generation $k$ such that if $v$ is deleted, then $\br$ no longer holds. The probability that $v$ is not deleted when passing to $X_{(1-\phi)(\la_1-\delta),(1-\phi)(\la_2-\delta)}$ is $(1-\phi)^k$.  The events $\overline\br$, $\rb_k\of\br$ and $\br\setminus (\rb_k\of\br)$ exhaust the sample space, and hence
\[
 \Prl{(1-\phi)(\la_1-\delta),(1-\phi)(\la_2-\delta)}{\br} \le \Prl{\plow}{\rb_k\of\br}+(1-\phi)^{k}.
\]
Since $\rb_k\of\br\subset \rb_{k+1}\of\br$, the sequence $\Prl{\plow}{\rb_k\of\br}$ is increasing. Taking the limit
of the inequality above,
\[
 \beta\bb{(1-\phi)(\la_1-\delta),(1-\phi)(\la_2-\delta)} = \Prl{(1-\phi)(\la_1-\delta),(1-\phi)(\la_2-\delta)}{\br} \le \lim_{k\to\infty} \Prl{\plow}{\rb_k\of\br}.
\]
Letting $\phi\to 0$, the lemma follows.
\end{proof}

It will often be convenient to {\em mark} some subset of the particles in generation $d$.
If $\ev$ is an event depending on the first $d$ generations, then we write
$\ev\of M$ for the event that $\ev$ holds after deleting all unmarked particles in generation
$d$. We write $\Pr^{\alpha}_{\porig}[\ev\of M]$ for the probability that $\ev\of M$ holds
when, given $\Xl{\porig}$, we mark the particles in generation $d$ independently with probability $\alpha$.
We suppress $d$ from the notation, since it will be clear from the event $\ev$.

Let
\[
 r(\porig,d,\alpha) = \Pr^{\alpha}_{\porig}[\rb_{d}\of M].
\]

\begin{lemma}\label{l2}
There exists a positive integer $d$ such that
\begin{equation*}
 r\bb{\pmed,d,\beta(\plow)} > \beta(\plow).
\end{equation*}
\end{lemma}
\begin{proof}
Let $X'$ be a branching process where the root vertex has $\Po(\delta-\xi)$ red and $\Po(\delta-\xi)$ blue offspring, and the descendants of these offspring are as in $\Xl{\plow}$. Then the branching process $Y$ created by merging an independent copy of $\Xl{\plow}$ and $X'$ at the root provides a lower coupling on $X_{\pmed}$, implying $r\bb{\pmed,d,\beta(\plow)}$ is at least the probability that $\rb_d\of \br$ holds in $Y$.

When $\rb_1\of \br$ holds in $X'$ then $\rb_d\of \br$ holds as well. Now $\rb_1\of \br$ holds in $X'$ if the root of $X'$ has at least two red and two blue offspring, each having property $\br$. Since each offspring of the root has $\br$ with probability $\alpha=\beta(\plow)$, $X'$ has property $\rb_1\of \br$ with probability $\zeta=\Ppo{\Po((\delta-\xi)\alpha)\ge 2}^2>0$.

On the other hand $\Xl{\plow}$ has $\rb_d\of\br$ with probability $\Prl{\plow}{\rb_d\of\br}$. Therefore the probability that $Y$ has $\rb_d\of \br$ is at least
$$ r_d=1- (1-\zeta)\bb{1-\Prl{\plow}{\rb_d\of\br}}.$$
By Lemma~\ref{l1}, as $d\to\infty$ we have
$\Prl{\plow}{\rb_d\of\br}\to \beta(\plow)=\alpha>0$, so
$$ r_d\to 1-(1-\zeta)(1-\alpha)>\alpha,$$
and there is a $d$ with $r_d\ge \alpha$, completing the proof.
\end{proof}

\emph{We fix the value of $d$, which satisfies Lemma~\ref{l2}, until the end of Section~\ref{sec:bptographs}.}

As the random graph model contains only a finite number of vertices, there is no equivalent for the event $\br$. In order to circumvent this we introduce an event $\el$, which depends only on the first $L$ generations of the branching process, such that conditional on $\el$ the probability that $\br$ holds is close to one. For a non-negative integer $k$ let $X[k]$ denote the first $k$ generations of the branching process~$X$.

\begin{lemma}\label{lem:inprob}
	There exists a positive integer $L$ and an event $\el$ depending only on the first $L$ generations of the branching process satisfying
	\begin{equation*}
	\Prl{\pmed}{\el} > \beta(\plow),
	\end{equation*}
	and if $\el$ holds then
	\begin{equation*}
	\Pr^{\beta(\pmed)}_{\pmed}\left[\br_L\of M \mid X[L]\right] \ge 1-4^{-3d}.
	\end{equation*}
	\end{lemma}

In the second condition (and again below) we are conditioning on the first $L$ generations of~$X$. An equivalent way of stating the second condition is, for all $\omega\in\el$, with $\Pr_{\pmed}[\omega]>0$, we have   
$$\Pr^{\beta(\pmed)}_{\pmed}\left[X\in \br_L\of M \mid X[L]=\omega[L]\right] \ge 1-4^{-3d}.$$

\begin{proof}[Proof of Lemma \ref{lem:inprob}]
	By monotonicity we have $\Prl{\pmed}{\br} 	>\beta(\plow).$
	Set
	$$\eta=\min\big\{\Prl{\pmed}{\br}-\beta(\plow),4^{-3d}\big\}>0.$$
	As $\br$ is measurable, there is an integer $L\ge d$ and an event $\el_1$ depending
	only on the first $L$ generations of the branching process such that
	$\Prl{\pmed}{\br\diff \el_1}\le \eta^2/2$,
	where $\diff$ denotes symmetric difference.
	Writing $1_\ev$ for the indicator function of an event $\ev$
	we have
	\begin{equation}\label{ebd}
	\eta^2/2 \ge \Prl{\pmed}{\br^\cc \cap \el_1}
	= \Es{\pmed}\big[ 1_{\el_1} \Prl{\pmed}{\br^\cc \mid X[L]}\big],
	\end{equation}
	where $\Es{\pmed}$ is the expectation corresponding to $\Prli{\pmed}$.
	Set
\begin{align*}
    \el &= \el_1 \cap \Bigl\{\Prl{\pmed}{\br \mid X[L]} \ge 1-\eta \Bigr\}\\
        &= \el_1 \cap \Bigl\{\omega\Bigm|\Prl{\pmed}{X\in\br \mid X[L]=\omega[L]} \ge 1-\eta \Bigr\},
\end{align*}
	and note that the event $\el$ depends only on the first $L$ generations of $X$.
	Since
	$$\Pr_{\pmed}\left[\br \mid X[L]\right]=
	\Pr^{\beta(\pmed)}_{\pmed}\left[\br_L\of M \mid X[L]\right]
	$$
	the second inequality in the statement of the lemma holds 
	and thus we only need to verify the first inequality.
	If $\el_1\setminus \el$ holds then $\Prl{\pmed}{\br^{\cc} \mid X[L]} \ge \eta$ leading to
	$$\Es{\pmed}\big[ 1_{\el_1} \Prl{\pmed}{\br^\cc \mid X[L]}\big]\ge \eta \Prl{\pmed}{\el_1\setminus \el}.$$
	Together with \eqref{ebd}, this implies $\Prl{\pmed}{\el_1\setminus\el}\leq\eta/2$ and hence
	\begin{align*}
	\Prl{\pmed}{\el} &\ge \Prl{\pmed}{\el_1}-\eta/2 \ge
	\Prl{\pmed}{\br}-\Prl{\pmed}{\br\diff \el_1}-\eta/2\\
	&\ge\Prl{\pmed}{\br}-\eta^2/2-\eta/2 \ge \beta(\plow),
	\end{align*}
	completing the proof.
\end{proof}

Fix an integer $L$ and an event $\el$ which satisfies the previous lemma. Let $\ca_0=\el$, and for $t\ge 1$ set $\ca_t=\rb_d\of\ca_{t-1}$.
Thus, $\ca_t$ is a `recursively robust' version of the event $\br_{dt}\of\el$, which depends on the first $dt+L$ generations of the branching process.

\begin{lemma}\label{lem:branchprob}
	For any $t\ge 0$, 
	\begin{equation*}
	\Prl{\pmed}{\ca_t}> \beta(\plow)
	\end{equation*}
	and if $\ca_t$ holds, we have 
	\begin{equation*}
	\Pr^{\beta(\pmed)}_{\pmed}\left[\br_{dt+L} \of M  \mid X[dt+L]\right] \ge 1-4^{-(2^t+2)d}.
	\end{equation*}
\end{lemma}

\begin{proof}
	The proof is by induction on $t$. The statement holds for $t=0$ by Lemma~\ref{lem:inprob}.
	
	As the descendants of different particles in generation $d$ of the branching process are independent we have
	$$\Prl{\pmed}{\ca_t} = r\bb{\pmed,d,\Prl{\pmed}{\ca_{t-1}}}.$$
	By the induction hypothesis we have $\Prl{\pmed}{\ca_{t-1}}> \beta(\plow)$ and the first statement follows from Lemma~\ref{l2}, as the function $r$ is monotone increasing in its last parameter.
	
	Now assume that if $\ca_{t-1}$ holds we have
	$$\Pr^{\beta(\pmed)}_{\pmed}\left[\br_{d(t-1)+L} \of M \mid X[d(t-1)+L]\right] \ge 1-4^{-(2^{t-1}+2)d}.$$
	Condition on $X[dt+L]$ and assume that $\ca_t$ holds.
	Since $\ca_t=\rb_d\of\ca_{t-1}$, there is a smallest set $Y$ of particles in generation $d$ such that $\ca_{t-1}$
	holds for each $y\in Y$, and $\rb_d\of M$ holds if we mark only the particles in $Y$.
	Since any tree witnessing $\rb_d$ contains a subtree witnessing $\rb_d$ in which each particle
	has at most two red and two blue offspring, we have $|Y|\le 4^d$.
	
	Mark each particle in generation $dt+L$ independently with probability
	$\beta(\pmed)$, and let $Y'$ be the set of particles $y\in Y$ for which
	$\br_{d(t-1)+L}\of M$ holds. 
	By the induction hypothesis and because the descendants of different particles in generation $d$ of the branching process are independent,
	each $y\in Y$ is included in $Y'$ independently with probability at least $1-4^{-(2^{t-1}+2)d}$. Thus
	\[
	\Pr(|Y\setminus Y'|\ge 2) \le \binom{|Y|}{2} \left(4^{-(2^{t-1}+2)d}\right)^2 \le 4^{2d}4^{-(2^t+4)d}
	= 4^{-(2^t+2)d}.
	\]
	By the definition of $\rb_d$, $\br_d$ holds whenever $|Y\setminus Y'|\le 1$
	and we only keep the particles in $Y'$ in generation $d$. 
	But then $\br_d\of\br_{d(t-1)+L}\of M=\br_{dt+L}\of M$ holds, proving the second statement.
\end{proof}

Let $T=T(n)$ satisfy $T=o(\log n)$ and $T/\log\log n\to \infty$ and set $s=1+dT+L$. Set $\ca= \br_1\of \ca_T$ and note that this event depends only on the first $s=1+dT+L$ generations of the branching process. The property $\ca$ plays a crucial role in identifying the vertices in the joint-giant. Our results from this section imply that
\begin{align}
\Prl{\phigh}{\ca}
&=\Ppo{\Po\left(\la_1\Prl{\phigh}{\ca_T}\right)>0}\Ppo{\Po\left(\la_2\Prl{\phigh}{\ca_T}\right)>0}\nonumber\\
&\stackrel{\mathclap{\mathrm{Lem. \ref{lem:branchprob}}}}{\ge}\>\>\, \Ppo{\Po\left((\la_1-\delta)\beta(\plow)\right)>0}
\Ppo{\Po\left((\la_2-\delta)\beta(\plow)\right)>0}\nonumber\\
&=\beta(\plow).\label{eq:alower}
\end{align}

We will need an additional property of the branching process in order to deduce that the vertices with property $\ca$ actually form a joint component. The exact relation between the property and connectivity is not obvious and will be clarified later. Let $\mathcal{J}$ be the event that the root of the branching process has a blue offspring with property $\br_s$ and it does not have a red offspring that has a blue offspring with property~$\ca$. 
\begin{lemma}\label{lem:brancomp}
We have
$$	\Prl{\phigh}{\mathcal{J}}\le \big(\Ppo{\Po(\la_2 \beta(\phigh))>0}+o(1)\big)\exp\big(-\la_1\Ppo{\Po(\la_2 \beta(\plow))>0}\big).$$
\end{lemma}

\begin{proof}
Let $\mathcal{J}_1$ be the event that the root of the branching process has a blue offspring with property $\br_s$ and $\mathcal{J}_2$ be the event that it does not have a red offspring that has a blue offspring with property~$\ca$.
Since $s\rightarrow \infty$ we have $\Prl{\phigh}{\br_s}=\Prl{\phigh}{\br}+o(1)=\beta(\phigh)+o(1)$. Therefore 
\begin{equation*}
\Prl{\phigh}{\mathcal{J}_1}=\Ppo{\Po(\la_2 \beta(\phigh))>0}+o(1)
\end{equation*}
and 
\begin{align*}
\Prl{\phigh}{\mathcal{J}_2}
&=\Pr\big[\Po\big(\la_1\Ppo{\Po(\la_2 \Prl{\phigh}{\ca})>0}\big)=0\big]\nonumber\\
&\stackrel{\eqref{eq:alower}}{\le}\exp\left(-\la_1\Ppo{\Po(\la_2 \beta(\plow))>0}\right).
\end{align*}
The result follows as the events $\mathcal{J}_1$ and $\mathcal{J}_2$ are independent.
\end{proof}

The final result we need for the branching process is that it is unlikely to become large quickly.

\begin{lemma}\label{lem:small}
For any positive integer $s=o(\log {n})$ we have
\begin{equation*}
\Pr_{\phigh}\left[|X[2s]|\ge n^{1/15}\right] = o(n^{-1}).
\end{equation*}
\end{lemma}

\begin{proof}
Let $\la=\la_1+\la_2$ and note that Lemma \ref{lem:epsexist} implies $\la>2$.  Denote by $N_t$ the number of particles in generation $t$ of the branching process~$X$.  The sum of Poisson random variables is Poisson, and hence $N_{t+1}$, conditioned on $N_t$, is distributed as $\Po(N_t\la)$.  We claim that 
\begin{equation}\label{eq:Ntinduct}
\Pr[N_{t+1}>2 N_t\lambda+2\log_2 n]\leq4n^{-2}.
\end{equation}
Then, except with probability $o(n^{-1})$, the inequality $N_{t+1}\leq2\lambda N_t+2\log_2 n$ holds for all $0\leq t\leq 2s$.  In that case, by induction on $t$ and noting $N_0=1$, we have $N_t\leq (t+1)(2\la)^t\log_2 n$ for all $0\leq t\leq 2s$.    The result follows (with any positive constant replacing $1/15$).

We establish inequality~\eqref{eq:Ntinduct} by direct calculation.  Let $k_0=\lceil 2N_t\la\rceil$. Then
\begin{align*}
\Pr[N_{t+1}>2N_t\lambda+2\log_2 n]&\leq \sum_{k=k_0+\lceil 2\log_2 n\rceil -2}^\infty\frac{(N_t\la)^ke^{-N_t\la}}{k!}\\
&\leq\sum_{k'=\lceil 2\log_2 n\rceil -2}^\infty\frac{(N_t\la)^{k_0+k'}e^{-N_t\la}}{k_0!\,k_0^{k'}}\\
&\leq\frac{(N_t\la)^{k_0}e^{-N_t\lambda}}{k_0!}\sum_{k'=\lceil 2\log_2 n\rceil -2}^\infty 2^{-k'}\\
&\leq 8n^{-2},
\end{align*}
completing the proof.
\end{proof}

\section{From the branching process to random graphs}\label{sec:bptographs}

Consider the random double graph $\tG=\tG(n,4s,\phigh)$, whose distribution is that of
$G(n,\la_1/n,\la_2/n)$ conditioned on the absence of any red-blue cycles of length at most $4s$, which includes the absence of any red-blue cycle of length 2, i.e., the girth of the merged edge set is larger than $4s$. Set $s=1+dT+L$ as in the previous section and recall that $s\to\infty$ and $s=o(\log{n})$.

Transferring from $G$ to $\tG$ has 
only a minor effect on the probability of local properties,
as sparse binomial random graphs are locally tree-like, and this also holds for the binomial double graph, when it is created from two sparse binomial random graphs. Roughly speaking this means that when exposing the edges in $\tG$ the probability that the next edge exposed is present should be close to the probability that the edge is present in $G(n,\la_1/n,\la_2/n)$, as long as the number of exposed edges remains small. We prove this result in the following lemma.

\begin{lemma}\label{l_small}
	Let $M_1$, $F_1$ be disjoint sets of possible edges over $V$ 
	and $M_2$, $F_2$ be disjoint sets of possible edges over $V$, with $|F_1|,|F_2|\le n^{2/3}$
	such that $F_1\cup F_2$ contains no cycle of length at most $4s$.
    Let $j\in\{1,2\}$ and 
	$e=\{w_1,w_2\}\not\in F_j\cup M_j$.
	Then for large enough $n$ we have
	$$ \Ppo{e\in E_j(\tG) \mid F_1\subseteq E_1(\tG) \subseteq M_1^\cc, F_2\subseteq E_2(\tG) \subseteq M_2^\cc}\le \la_j/n.$$
	If in addition $\{w_2\}$ is disjoint from any edge in $F_1\cup F_2$ then
	$$(1-n^{-1/4})\la_j/n \le \Ppo{e\in E_j(\tG) \mid F_1\subseteq E_1(\tG) \subseteq M_1^\cc, F_2\subseteq E_2(\tG) \subseteq M_2^\cc}.$$
\end{lemma}

\begin{proof}
	Without loss of generality assume $j=1$.
	Let $G=G(n,\la_1/n,\la_2/n)$. Let $\mathcal{F}$ denote the event $F_1\subseteq E_1(G) \subseteq M_1^\cc, F_2\subseteq E_2(G) \subseteq M_2^\cc$. In addition let $\cD$ be the event that no cycle of length at most $4s$ is found in $G$. Clearly
	$$\Ppo{e\in E_1(\tG) \mid F_1\subseteq E_1(\tG) \subseteq M_1^\cc, F_2\subseteq E_2(\tG) \subseteq M_2^\cc}=\Ppo{e\in E_1(G) \mid \cD,\mathcal{F}}.$$
	
	Note that, conditional on $\mathcal{F}$, for $i=1,2$ the edges in $E_i(G)$ except those in $F_i$ and $M_i$ appear independently not only of the other edges in $E_i$, but also of the edges in $E_{3-i}$. Since the event $e\in E_1(G)$ is increasing and the event $\cD$ is decreasing we have, by Harris's Lemma (\cite{MR0115221}),
	$$\Ppo{e\in E_1(G) \mid \cD,\mathcal{F}}=\frac{\Ppo{e\in E_1(G),\cD\mid \mathcal{F}}}{\Ppo{ \cD \mid\mathcal{F}}}\le \frac{\Ppo{e\in E_1(G)\mid \mathcal{F}}\Ppo{\cD\mid \mathcal{F}}}{\Ppo{ \cD \mid\mathcal{F}}}=\frac{\la_1}{n},$$
	proving the upper bound.
	
	Now for the lower bound.
	Let $\mathcal{P}$ be the event that there is no red-blue path between $w_1$ and $w_2$ in $G$, with length between 1 and $4s-1$. Note that conditional on $\cD,\mathcal{P},\mathcal{F}$ the edge $e$ is present in $E_1$ with probability $\la_1/n$.
	Therefore 
\begin{align*}
\Ppo{e\in E_1(G) \mid \cD,\mathcal{F}}&\ge \Ppo{e\in E_1(G) ,\mathcal{P} \mid \cD,\mathcal{F}}\\
&= \Ppo{e\in E_1(G) \mid \cD,\mathcal{P},\mathcal{F}}\Ppo{\mathcal{P}\mid \cD,\mathcal{F}}\\
&=\frac{\la_1}{n}\Ppo{\mathcal{P}\mid \cD,\mathcal{F}}.
\end{align*}
	
	All that remains to show is that $\Ppo{\mathcal{P}\mid \cD,\mathcal{F}}\ge (1-n^{-1/4})$. Note that both the event $\mathcal{P}$ and $\cD$ are decreasing, therefore Harris's Lemma implies $\Ppo{\mathcal{P}\mid \cD,\mathcal{F}}\ge \Ppo{\mathcal{P}\mid \mathcal{F}}$.
	Now consider a path of length between 1 and $4s-1$ between $w_1$ and $w_2$. Any such path must contain a subpath, where none of the edges is contained in $F_1$ or $F_2$. In addition, for one of these subpaths, the first vertex of this subpath is $w_2$, while the last vertex of this subpath must be in the set $W$, which consists of the vertex $w_1$ and the set of endpoints of the edges in $F_1\cup F_2$. Since none of the edges in the subpath is contained in $F_1$ or $F_2$, it is the case that each red-blue edge appears independently with probability at most $(\la_1+\la_2)/n$.
	Now for $\mathcal{P}$ to fail one such subpath must be present, but the probability of this event is at most 
	\[
	\sum_{\ell=1}^{4s-1}  |W| n^{\ell-1}\left(\frac{\la_1+\la_2}{n}\right)^{\ell}
	\le \frac{5n^{2/3}}{n} \sum_{\ell=1}^{4s-1} (\la_1+\la_2)^\ell \le 5n^{-1/3}O(1)^{o(\log n)} \le n^{-1/4}
	\]
	if $n$ is large enough. Hence, the probability that such a subpath is present
	is at most $n^{-1/4}$, and so is the probability that $\mathcal{P}$ fails, completing the proof.
\end{proof}

Let $\tG_v[t]$ be the subgraph of $\tG$ formed by the vertices within distance $t$ of $v$, noting that for $t\le 2s$ this graph is by definition a tree. For an event $\mathcal{E}$ of the branching process, which depends only on the first $t\le 2s$ generations we say that a vertex $v\in V(\tG)$ has property $\mathcal{E}$ if $\tG_v[t]$ has property $\mathcal{E}$, when viewed as a branching process rooted at $v$.
In order to couple $\tG_v[t]$ to a branching process we will use the following exploration process.

Let $W$ be a set of vertices in $\tG$ and $t\le 2s$. We explore the neighbourhood of the vertices $W$ until depth $t$ in the following manner.
Initially we set every vertex in $W$ `active' and all other vertices `untested'. We perform the following for each vertex $w\in W$ one after the other.
In each step we pick an `active' vertex $u$ closest to $w$, and expose one by one the edges between $u$ and the 'untested' vertices.
The newly discovered neighbours of $u$ are set as `active', while the state of $u$ is changed to 'tested'. We abandon the exploration process associated to a given $w$ if there are no active vertices at distance less than $t$ from $w$ or we reach $n^{1/15}$ vertices.
We refer to this as running the exploration process on $W$ until depth $t$. 

Note that, we are essentially performing a breadth-first search starting at each $w\in W$, excluding the vertices in $W$ and the previously discovered vertices at every stage of the exploration process. 

This exploration process creates a forest, where each tree within the forest contains exactly one vertex from $W$. Denote by $T_w$ the graph discovered during the exploration process associated to $w$. (If the exploration is abandoned, $T_w$ is whatever has been discovered at the point of abandonment.) 

\begin{lemma}\label{lem:coupling}
	Let $W\subseteq V(\tG)$ such that $|W|\le n^{3/5}$ and $t\le 2s$. Run the exploration process on $W$ until depth $t$. Then for any $w\in W$ we may couple each $T_w$ to agree with $X_{\phigh}[t]$ with probability $1-o(1)$.
\end{lemma}

\begin{proof}
Note that at the end of each step of the exploration process there are no edges incident to `untested' vertices.
Therefore, when selecting an active vertex $u$, by Lemma~\ref{l_small}, as long as we have reached at most $n^{2/3}$ vertices,
conditional on everything so far each red edge is present with probability
$(1+O(n^{-1/4}))\la_1/n$, while each blue edge is present with probability $(1+O(n^{-1/4}))\la_2/n$.
As the number of untested vertices is $n-O(n^{2/3})$,
we may couple the number of new red and blue neighbours of $u$ found with a Poisson distribution
with mean $\la_1$ and $\la_2$ respectively so that the two numbers agree with probability $1-O(n^{-1/4})$.

By Lemma~\ref{lem:small} \whp, $\Xl{\phigh}[2s]$ contains at most $n^{1/15}$ particles, implying that
\whp\ the exploration process for each $w$ also contains at most $n^{1/15}$ vertices and the result follows.
\end{proof}

For a double graph $G$ on $n$ vertices let $U(G)$ be the maximal subset of vertices of $G$ such that in the subgraph spanned by $U$ every vertex is found in both a red and a blue tadpole graph (a tadpole graph is a cycle and a path joined at a vertex, in our case the path may be empty). Note that $U(G)$ is well defined and unique as it can be created by repeatedly removing any vertex which is not contained in both a red and blue tadpole within the graph spanned by the remaining vartices.

In the following lemma we show that \whp\ the number of vertices with property $\ca$ is large and every vertex with property $\ca$ is contained in $U(\tG)$.

\begin{lemma}\label{lem:size}
The number of vertices of $\tG$ with property $\ca$ is \whp\ at least $\beta(\plow)n$. In addition \whp\ every vertex with property $\ca$ is in $U(\tG)$.
\end{lemma}

\begin{proof}
Let $v$ be a vertex of $\tG$, and run the exploration process on $\{v\}$ until distance $2s$.
Lemma~\ref{lem:coupling} implies that $\tG_v[2s]$ and $\Xl{\phigh}[2s]$ can be coupled as to agree in the natural sense \whp. Therefore $v$ has property $\ca$ with probability $\Prl{\phigh}{\ca}+o(1)$. 

Write $A$ for the set of vertices with property $\ca$. Clearly $\Exp{|A|}=\Prl{\phigh}{\ca}n+o(n)$.
By Lemma~\ref{l_small} the probability that two vertices are within distance $2s$ is at most
$$\sum_{i=1}^{2s}n^{i-1}\left(\frac{\la_1}{n}+\frac{\la_2}{n}\right)^{i}=o(1),$$
implying $\Var(|A|)=o(n^2).$  
Hence (by Chebyshev's inequality) and \eqref{eq:alower}, $|A|/n$ converges in probability
to
$$
\Prl{\phigh}{\ca}\ge\beta(\plow).
$$

All that is left to show is that \whp\ every vertex with property $\ca$ is in $U$. Condition on $v$ having property $\ca$. Let $u_1,\ldots,u_k$ be all the offspring of~$v$ with property $\ca_T$.  Since $v$ has property~$\ca$ we know that $k\geq2$ and that there is at least one red offspring and one blue offspring within $u_1,\ldots,u_k$.  We do not commit ourselves to a particular choice of red or blue offspring at this stage, as we need some flexibility later.

Similarly as before run the exploration process on $\{v\}$ until depth $2s$. Note that by Lemma~\ref{lem:small} the exploration process is abandoned with probability $o(n^{-1})$ and this bound also holds when conditioning on $v$ having property $\ca$ as this event has probability $\Theta(1)$.

Choose any offspring $u_i$ and let $S_i$ be the set of vertices in generation $s-1$ in the subtree rooted at~$u_i$.
We now examine the vertices $w\in S_i$ one by one and mark each of these vertices if it has property $\ca$ in the tree $\tG_v[2s]$ rooted at $v$.
By Lemma~\ref{lem:coupling} we can couple $\tG_v[2s]$ to the branching process $\Xl{\phigh}[2s]$ with probability $1-o(1)$. 
Note that conditioning on $v$ having property $\ca$ affects only the first $s$ generations of the branching process.
Therefore the subtree rooted at $w$ can be coupled to the branching process $\Xl{\phigh}[s]$ with probability $1-o(1)$. Hence 
the probability that we mark $w$ is at least $\Prl{\phigh}{\ca}-o(1)$. 

Note that $\Prl{\phigh}{\ca_T}$ is bounded away from $\Prl{\pmed}{\ca_T}$ and thus
\begin{align*}
\Prl{\phigh}{\ca}-o(1)
&=\Ppo{\Po\left(\la_1\Prl{\phigh}{\ca_T}\right)>0}\Ppo{\Po\left(\la_2\Prl{\phigh}{\ca_T}\right)>0}-o(1)\\
&\ge\Ppo{\Po\left(\la_1\Prl{\pmed}{\ca_T}\right)>0}\Ppo{\Po\left(\la_2\Prl{\pmed}{\ca_T}\right)>0}\\
&\stackrel{\mathclap{\mathrm{Lem. \ref{lem:branchprob}}}}{\ge}\>\>\, \Ppo{\Po\left(\la_1\beta(\plow)\right)>0}\Ppo{\Po\left(\la_2\beta(\plow)\right)>0}\\
&\stackrel{\eqref{eq:xi}}{\ge} \Ppo{\Po\left((\la_1-\xi)\beta(\pmed)\right)>0}\Ppo{\Po\left((\la_2-\xi)\beta(\pmed)\right)>0}\\
&=\beta(\pmed).
\end{align*}
In summary we can view each
$w\in S_i$ as marked independently with probability (at least) $\beta(\pmed)$.

Now, in the tree $\tG_v[s]$ rooted at $v$, every offspring $u_i$ has property $\ca_T$, and every vertex at depth $dT+L$ starting from $u_i$ has property $\ca$ with probability at least $\beta(\pmed)$. 
Since $T/\log\log n\to \infty$, we have $4^{-(2^T+2)d}\le n^{-3}$ if $n$ is large enough. 
Hence, from Lemma~\ref{lem:branchprob}, with probability $1-o(n^{-1})$ all of the vertices $u_1,\ldots,u_k$ have property $\br_{dT+L} \of \ca$.  Note that the failure probability is small enough that this situation holds whp uniformly over vertices $v$ satisfying~$\ca$.  Suppose that this is the case. Then, certainly, $v$ has property $\br_{1+dT+L} \of \ca$.  More than that, the fact that $v$ has property $\br_{1+dT+L} \of \ca$ is witnessed by any choice of a red offspring~$u_i$ and blue offspring~$u_j$.   For each $i,j\in[k]$ with $u_i$ red and $u_j$ blue, choose a 
minimal subtree of~$\tG$ containing $u_i$ and $u_j$ witnessing that $v$ has property $\br_{1+dT+L} \of \ca$.  
(A minimal subtree is 
a balanced binary red-blue tree of depth $1+dT+L$.)  
Let the collection of all such subtrees, ranging over feasible $i,j\in[k]$, be denoted~$\cT_v$.
Note that if $\tau\in\cT_v$ is any such subtree, then each leaf of $\tau$ has property~$\ca$ (in the graph $\tG$ excluding the vertices of $\tau$).

Let $U^\circ=\bigcup_v \bigcup_{\tau\in\cT_v}V(\tau)$, where the first union is over all vertices~$v$ with property $\ca$.  
With probability $1-n\,o(n^{-1})=1-o(1)$, we have that
$|\cT_v|>0$ for every vertex $v$ with property~$\ca$.  It follows that all vertices with property $\ca$ are contained in $U^\circ$.  On the other hand, it is not too difficult to see that $U^\circ\subseteq U(\tG)$.
Take any vertex $u\in U^\circ$.  By definition of $U^\circ$ we must have $u\in V(\tau)$ for some $\tau\in\cT_v$.  
Trace a red path in $\tau$ from~$u$ to a leaf~$w$ of~$\tau$.  Now pick a suitable tree from $\cT_w$ (i.e., one that shares only vertex $w$ with~$\tau$) and trace a red path in it from $w$ to a leaf~$w'$.  Note that we have included sufficiently many trees in $\cT_w$ that we can avoid using the parent of~$w$.
Then repeat, tracing a red path from~$w'$, etc.  This process will terminate when the path intersects itself, at which point we have a red tadpole.  Clearly, the same construction works for blue tadpoles.
\end{proof}

Lemma~\ref{lem:size} implies that \whp\ $A$, the set of vertices with property $\ca$, is a subset of~$U(\tG)$. Next we will show that $U(\tG)\subseteq B_{s}$, where $B_{s}$ is the set of vertices with property $\br_s$.

\begin{claim}\label{clm:UinB}
	For every $v\in U(\tG)$ we have that $v$ has property $\br_s$.
\end{claim}

\begin{proof}
We will show that that $v$ has property $\br_s$ already within the subgraph of $\tG$ spanned by $U(\tG)$.
Every vertex $v$ in $U(\tG)$ lies within a red tadpole;  choose a tadpole for each~$v$ and form the union of tadpoles over all vertices $v\in U(\tG)$.  The resulting graph has the property that each of its connected components has at least one cycle.  We can therefore choose a red subgraph $(U(\tG),F_r)$ of $\tG$ such that every connected component of $(U(\tG),F_r)$ is unicyclic.  Orient the edges in each cycle of $(U(\tG),F_r)$ consistently, and orient all other other edges towards the unique cycle in their component.  Repeat the process to obtain a blue subgraph $(U(\tG),F_b)$ together with an orientation.

Let $v$ be any vertex in $U(\tG)$.  There is a unique oriented red edge $(v,u_r)$ leaving $v$ and a unique oriented blue edge $(v,u_b)$ leaving $v$.  Make $u_r$ and $u_b$ the offspring of $v$.  There are unique red and blue oriented edges leaving $u_r$ and $u_b$, so the process can be repeated.  The choice of orientations for the edges of $(U(\tG),F_r)$ and $(U(\tG),F_b)$ ensures that the process never gets stuck. In addition up to depth $s$, no cycles are created, as the girth of $\tG$ is at least $4s$.  Thus there is a complete red-blue binary tree of depth $s$ rooted at~$v$, witnessing the fact that $v$ has property~$\br_s$.  
\end{proof}

In order to prove Proposition~\ref{prop:size2} we need to show that every red component and every blue component within $U(\tG)$ has size at least $n^{3/5}$. Consider a red component within $U(\tG)$. 
Any such component must contain a connected unicyclic spanning subgraph. In addition every vertex within the component must have a blue neighbour in~$U(\tG)$, and no vertex in the component may have a red neighbour outside of $U(\tG)$ that has a blue neighbour in~$U(\tG)$.

The condition of being a component of $U(\tG)$, being a global one, is difficult to deal with.  Therefore
we look first at an event that is closely related, but which refers only to local conditions.  
For $W\subset V$ let $\mathcal{C}_r(W)$ be the event that
\begin{itemize}
	\item $W$ contains a red connected unicyclic spanning subgraph;
	\item every vertex in $W$ has a blue neighbour in $B_s$;
	\item no vertex in $W$ has a red neighbour in $V\setminus W$ with a blue neighbour in $A$.
\end{itemize}

The event $\mathcal{C}_b(W)$ is defined analogously for blue components, i.e., the two colours are swapped. Set $\mathcal{C}$ as the event that there exists $W\subseteq V$ with $T\le |W|\le n^{3/5}$ such that either $\mathcal{C}_r(W)$ or $\mathcal{C}_b(W)$ holds.  Intuitively, $\mathcal C$ is a necessary local condition for connected components of intermediate size $T\le k\le n^{3/5}$ to exist.  We show in the next lemma that $\mathcal C$ is a low probability event.  After that, we just need to relate the local event to the global one that is of actual interest.   

\begin{lemma}\label{lem:nosmallcomp}
Whp the event $\mathcal{C}$ fails to hold  in $\tG=\tG(n,\la_1,\la_2)$.
\end{lemma}

\begin{proof}
We will examine the probability of the events in the definition of $\mathcal{C}_r(W)$ one by one when $W$ contains $k$~vertices. 
Denote by $\mathcal{U}_r(W)$ the event that $W$ contains a red connected unicyclic spanning subgraph.
The number of connected unicyclic graphs on $k$ vertices is at most $k^k$, as there are $k^{k-2}$ ways to select a spanning tree on $k$~vertices and fewer than $k^2$ ways to select an additional edge.
Therefore by Lemma~\ref{l_small} the probability that $\mathcal{U}_r(W)$ holds is at most
$$\Ppo{\mathcal{U}_r(W)}\leq k^k \left(\frac{\la_1}{n}\right)^k.$$
Note that until this point we have only exposed red edges in~$W$.

Denote by $\tG_w'[s+2]$ the subtree of $\tG_w[s+2]$ created by removing every red neighbour of $w$ within $W$ and any descendents these vertices might have. 
Recall that $\mathcal{J}$ is the event that the root of the branching process has a blue offspring with property $\br_s$ and it does not have a red offspring that has a blue offspring with property~$\ca$. Therefore the event that $w$ has a blue neighbour in $B_s$ and $w$ does not have a red neighbour in $V\setminus W$ with a blue neighbour in $A$ is equivalent to the event that $\mathcal{J}$ holds for $\tG_w'[s+2]$. 

Now run the exploration process on $W$ until depth $s+2$.
Recall that $T_w$ is the tree discovered from vertex~$w$, and since the exploration process doesn't explore any edge between the initially active vertices
we have that $T_w\subseteq \tG_w'[s+2]$. 
Then the event $\tG_w'[s+2]$ has property $\mathcal{J}$ is contained in the union of events $T_w$ has property $\mathcal{J}$ and the event $T_w\neq \tG_w'[s+2]$. 
Now $T_w\neq \tG_w'[s+2]$ can occur in two different ways, either $T_w$ was abandoned or there is an unexposed edge adjacent to a tested vertex in $T_w$ other than the red edges between $w$ and $W\setminus \{w\}$.
Denote by $\mathcal{J}_w$ the union of the event that $T_w$ has property $\mathcal{J}$ and the event that $T_w$ was abandoned. In addition let $\mathcal{K}_w'$ be the event that there is an unexposed edge adjacent to a tested vertex in $T_w$ other than the red edges between $w$ and $W\setminus \{w\}$.

The discussion so far is summarised in the following inclusion:
\begin{equation}\label{eq:Crbound'}
\mathcal{C}_r(W)\subseteq 
\mathcal{U}_r(W)\cap \bigcap_{w\in W}\big(\mathcal{J}_w\cup \mathcal{K}'_w\big).
\end{equation}

Recall that in the exploration process we expose every edge between the currently selected active vertex and all untested vertices. Therefore any unexposed edge adjacent to a tested vertex in $T_w$ must be between $T_w$ and $T_{w'}$ for some $w'\in W \setminus\{w\}$.
Consider an auxiliary graph $H$ with vertex set $W$, where two vertices are connected if there is an edge between the corresponding exploration processes other than a red edge connecting the roots. 

Note that the edges we consider between $T_w$ and $T_{w'}$ are either not present or have not been exposed. 
Therefore by Lemma~\ref{l_small} the probability that there is an edge between $T_w$ and $T_{w'}$ is at most $n^{2/15}(\la_1/n+\la_2/n)\le n^{-4/5}$. Therefore $G(W,n^{-4/5})$ provides an upper coupling on $H$ independently of the result of the exploration processes.

For some $W'\subseteq W$ denote by $\mathcal{K}(W')$ the event that in $G(W,n^{-4/5})$ the set of vertices with positive degree corresponds to $W'$.
Therefore 
\begin{equation}\label{eq:Crbound}
\bigcap_{w\in W}\big(\mathcal{J}_w\cup \mathcal{K}'_w\big)\subseteq\bigcup_{W'\subseteq W}\left[\mathcal{K}(W')\cap \bigcap_{w\in W\setminus W'}\mathcal{J}_w\right].
\end{equation}

Note that $\mathcal{K}(W')$ is independent of $\bigcap_{w\in W\setminus W'}\mathcal{J}_w$ implying
$$\Ppo{\mathcal{K}(W')\cap \bigcap_{w\in W\setminus W'}\mathcal{J}_w}
=\Ppo{\mathcal{K}(W')}\Ppo{\bigcap_{w\in W\setminus W'}\mathcal{J}_w}.$$
Now $\mathcal{K}(W')$ is contained in the event that $G(W',n^{-4/5})$ contains at least $|W'|/2$ edges which has probability at most
$$\Ppo{\mathcal{K}(W')}\le\binom{|W'|^2/2}{|W'|/2}\left(n^{-4/5}\right)^{|W'|/2}\le \Big(e|W'|n^{-4/5}\Big)^{|W'|/2}\le n^{-|W'|/12}.$$
In addition Lemma~\ref{lem:brancomp}, Lemma~\ref{lem:small} and Lemma~\ref{lem:coupling} imply
$$\Ppo{\bigcap_{w\in W\setminus W'}\mathcal{J}_w}=\Big(\Ppo{\Po(\la_2 \beta(\phigh))>0}\exp\big(-\la_1\Ppo{\Po(\la_2 \beta(\plow))>0}\big)+o(1)\Big)^{|W\setminus W'|}.$$

Therefore 
\begin{align}
&\Ppo{\bigcup_{W'\subseteq W}\left(\mathcal{K}(W')\cap \bigcap_{w\in W\setminus W'}\mathcal{J}_w\right)}\notag\\
&\le \sum_{\ell=0}^{k}\binom{k}{\ell}n^{-\ell/12}\Big(\Ppo{\Po(\la_2 \beta(\phigh))>0}\exp\big(-\la_1\Ppo{\Po(\la_2 \beta(\plow)>0}\big)+o(1)\Big)^{k-\ell}\notag\\
&= \Big(\Ppo{\Po(\la_2 \beta(\phigh))>0}\exp\big(-\la_1\Ppo{\Po(\la_2 \beta(\plow)>0}\big)+o(1)+n^{-1/12}\Big)^k\notag\\
&=\Big(\Ppo{\Po(\la_2 \beta(\phigh)>0}\exp\left(-\la_1\Ppo{\Po(\la_2 \beta(\plow))>0}\right)+o(1)\Big)^k,\label{eq:long}
\end{align}
where in the penultimate equality we used the binomial theorem and in the last equality we use $n^{-1/12}=o(1)$.

Putting together \eqref{eq:Crbound'}, \eqref{eq:Crbound},  and \eqref{eq:long}, and recalling that the event $\mathcal{U}_r(W)$ is independent of the other events,
\begin{align}
&\sum_{\substack{W\subseteq V\\T\le |W|\le n^{3/5}}}\Ppo{\mathcal{C}_r(W)}\nonumber\\
&\le \sum_{k=T}^{n^{3/5}} \binom{n}{k} k^k\left(\frac{\la_1}{n}\right)^{k}
\Big(\Ppo{\Po(\la_2 \beta(\phigh))>0}\exp\left(-\la_1\Ppo{\Po(\la_2 \beta(\plow))>0}\right)+o(1)\Big)^k\nonumber\\
&\le \sum_{k=T}^{n^{3/5}} 
\Big(e\la_1\Ppo{\Po(\la_2 \beta(\phigh))>0}\exp\left(-\la_1\Ppo{\Po(\la_2 \beta(\plow))>0}\right)+o(1)\Big)^k\nonumber\\
&\le \sum_{k=T}^{n^{3/5}} 
\Big(e\la_1\Ppo{\Po(\la_2 \beta(\phigh))>0}\exp\left(-\la_1\Ppo{\Po(\la_2 (\beta(\phigh)-\eps))>0}\right)+o(1)\Big)^k \label{eq:probofC}
\end{align}
where the last inequality follows from \eqref{eq:deltaspec}. 
By \eqref{eq:epssmall} we have
$$e\la_1\Ppo{\Po(\la_2 \beta(\phigh))>0}\exp\left(-\la_1\Ppo{\Po(\la_2 (\beta(\phigh)-\eps))>0}\right)<1,$$
and thus \eqref{eq:probofC} is $o(1)$ when $n$ is large enough.
A similar bound holds for the sum of $\Ppo{\mathcal{C}_b(W)}$.
Since 
$$
\mathcal{C}=\bigcup_{\substack{W\subseteq V\\T\le |W|\le n^{3/5}}}\!\!\mathcal{C}_r(W)\cup\mathcal{C}_b(W)
$$
the result follows.
\end{proof}

Now we have everything needed to prove Proposition~\ref{prop:size2}.

\begin{proof}[Proof of Proposition~\ref{prop:size2}]
Note that the event $|U'(G)|\ge x$ is monotone increasing for any $x\ge 0$. Since not having a cycle of length at most $4s$ is a decreasing event, by Harris's Lemma we have
$$\Ppo{|U'(\tG(n,\la_1/n,\la_2/n))|\ge x}\le \Ppo{|U'(G(n,\la_1/n,\la_2/n))|\ge x}.$$

Next we will compare the sets $U(\tG)$ and~$U'(\tG)$.
For $W\subset V$ let $\mathcal{C}_r'(W)$ be the event that

\begin{itemize}
	\item $W$ contains a red connected unicyclic spanning subgraph;
	\item every vertex in $W$ has a blue neighbour in $U(\tG)$;
	\item no vertex in $W$ has a red neighbour in $V\setminus W$ with a blue neighbour in $U(\tG)$.
\end{itemize}
(This is the global version of the event analysed in Lemma~\ref{lem:nosmallcomp}.)
In addition let $\mathcal{C}_b'(W)$ be the analogous events for blue components, i.e., the two colours are swapped. Let $\mathcal{C}'$ be the event that there exists $W\subseteq U(\tG)$ with $T\le |W| \le n^{3/5}$ such that either $\mathcal{C}_r'(W)$ or $\mathcal{C}_b'(W)$ holds.
Noting that every component in $U(\tG)$ has size at least~$T$, event $\overline{\mathcal{C}'}$ implies $U'(\tG)\supseteq U(\tG)$.  
 
By Lemma~\ref{lem:size} and Claim~\ref{clm:UinB}, \whp, $A\subseteq U(\tG)\subseteq B_s$. These inclusions imply $\mathcal{C}_r'(W)\subseteq \mathcal{C}_r(W)$ 
and $\mathcal{C}_b'(W)\subseteq \mathcal{C}_b(W)$, which in turn imply $\mathcal{C}'\subseteq \mathcal{C}$. 
By Lemma~\ref{lem:nosmallcomp}, $\overline{\mathcal{C}}$ holds \whp, and hence $\overline{\mathcal{C}'}$ also holds \whp.

Lemma~\ref{lem:size} states that \whp\ $|U(\tG)|\ge \beta(\plow)n$, from which
$$|U'(G)|\ge |U(G)|\ge\beta(\plow)n \stackrel{\eqref{eq:deltaspec}}{\ge} \beta(\phigh)n-\eps n,$$
\whp, and the result follows as this inequality holds for arbitrary $0<\eps<\eps_0$.
\end{proof}

\section{Proofs of Theorem~\ref{thm:main} and \ref{thm:aux}}\label{sec:proofs}

Let $G_v[t]$ be the subgraph of $G$ formed by vertices at distance at most $t$ from $v$. 

\begin{claim}\label{clm:gianthasB}
	If $v$ is in a joint component of size larger than one then for every $t>0$ either $G_v[t]$ contains a cycle, or $G_v[t]$ is a tree and $G_v[t]$ has property $\br_t$ when viewed as a branching process rooted at $v$.
\end{claim}

\begin{proof}
Let $J$ be the subgraph of $G$ spanned by the joint component of $v$. The result follows if we can show that if $J_v[t]$ is a tree then $J_v[t]$ has property $\br_t$, when viewed as a branching process rooted at $v$.
We will give a procedure for marking vertices in the tree $J_v[t]$ that terminates with a complete red-blue binary subtree in $J_v[t]$ of depth $t$ being marked.  This subtree is a witness to $v$ having property $\br_t$.

First mark $v$.  Since $J$ is connected in the red graph, there must be a red path from $v$ to some leaf of $J_v[t]$.  Let the first vertex on this path be $u_r$.  Similarly, let $u_b$ be the first vertex in some blue path from $v$ to a leaf of $J_v[t]$.  Mark $u_r$ and $u_b$.  By construction, there is a red path from $u_r$ to a leaf of $J_v[t]$ lying completely in the subtree of $J_v[t]$ rooted at~$u_r$.  
Also, since $J$ is connected in the blue graph, there must also be a blue path from $u_r$ to a leaf of $J_v[t]$, and this path necessarily lies within the subtree of $J_v[t]$ rooted at~$u_r$.  A similar argument applies to the vertex~$u_b$.  
The situation at $u_r$ and $u_b$ replicates the situation that existed at the root $v$ of $J_v[t]$,
so we can continue the marking process until we reach the leaves.  The result is a witness to~$v$ having property $\br_t$.
\end{proof}

\begin{proof}[Proof of Theorem~\ref{thm:main}]
For the lower bound our aim is to show that for every $\eps>0$ \whp\ the size of the joint-giant is at least $\beta(\porig)n-2\eps n$.  If $\beta(\porig)=0$ there is nothing to prove.
Recall that $\beta(\la_1,\la_2)$ is continuous, therefore for every $\eps>0$ there exists a $\delta$ such that $\beta(\la_1-\delta,\la_2-\delta)>\beta(\la_1,\la_2)-\eps$ and $\beta(\la_1-\delta,\la_2-\delta)\in \mathbb{R}^2\setminus C$.

We will expose the edges of $G(n,\la_1/n,\la_2/n)$ in two rounds. First we expose the edges of $G(n,(\la_1-\delta)/n,(\la_2-\delta)/n)$ and then merge this graph with a copy of $G(n,\delta/n,\delta/n)$, which provides a lower coupling for $G(n,\la_1/n,\la_2/n)$. 

Expose all the edges in $G(n,(\la_1-\delta)/n,(\la_2-\delta)/n)$. Applying Proposition~\ref{prop:size2} to $G(n,(\la_1-\delta)/n,(\la_2-\delta)/n)$ we have \whp\ that
$$|U'(G(n,(\la_1-\delta)/n,(\la_2-\delta)/n))|\ge \beta(\la_1-\delta,\la_2-\delta)n-o(n)\ge (\beta(\porig)-2\eps)n,$$
for large enough $n$.

Recall that every red component in $U'(G(n,(\la_1-\delta)/n,(\la_2-\delta)/n))$ has size at least $n^{3/5}$ and note that there are at most $n$ such components. Therefore the probability that there exists a pair of red components in $U'(G(n,(\la_1-\delta)/n,(\la_2-\delta)/n))$ with no red edge between these components in $G(n,\delta/n,\delta/n)$ is at most
$$ n^{2}(1-\delta/n)^{n^{6/5}}\le n^2 \exp\left(-\delta n^{1/5}\right)=o(1).$$
An analogous proof for the blue graph completes the proof of the lower bound.

Now for the upper bound. Let $s\to \infty$ satisfy $s=o(\log{n})$. Claim~\ref{clm:gianthasB} implies that if $v$ is in a linear sized joint component then either $G_v[s]$ contains a cycle, which occurs with probability at most
$$\sum_{\ell=2}^{2s} n^{\ell-1}\left(\frac{\la_1+\la_2}{n}\right)^{\ell}= o(1)$$ 
or $G_v[s]$ is cycle-free and when viewed as a branching process rooted at $v$ it has property $\br_s$, which occurs with probability at most $\beta(\porig)+o(1)$. Write $N$ for the number of vertices satisfying one of these two conditions. Then $\Exp{N}=\beta(\porig)n+o(n)$. Recall that the probability that a pair of vertices are within distance $2s$ is $o(1)$, implying that $\mathrm{Var}(N)=o(n^2)$, and thus by Chebyshev's inequality the number of such vertices is concentrated around its expectation providing the upper bound.

We have just proven that \whp\ the number of vertices in the joint-giant matches the number of vertices in components of linear size up to an additive term of $o(n)$, implying that there is a unique linear sized joint component.
\end{proof}

\begin{proof}[Proof of Theorem~\ref{thm:aux}]
	Note that any joint component must contain both a red and a blue spanning tree. The probability that there exists a component of size between $3$ and $\eps n$ is at most
	\begin{align*}
	\sum_{k=3}^{\eps n} \binom{n}{k}k^{2k-4} \left(\frac{\la_1}{n}\right)^{k-1} \left(\frac{\la_2}{n}\right)^{k-1} 
	& \le \sum_{k=3}^{\eps n} \frac{n^2}{\la_1\la_2 k^4}\left(\frac{en}{k}\right)^k k^{2k}\left(\frac{\la_1}{n}\right)^k \left(\frac{\la_2}{n}\right)^k\\
	&= \sum_{k=3}^{\eps n} \frac{n^2}{\la_1\la_2 k^4}\left(\frac{e\la_1 \la_2 k}{n}\right)^k=o(1)
	\end{align*}
	when $\eps< (e \la_1\la_2)^{-1}$.
\end{proof}

\section*{Acknowledgement}
We thank an anonymous referee for carefully reading the manuscript and suggesting improvements.

\bibliographystyle{plain}
\bibliography{References}

\end{document}